\documentclass[12pt]{article}
\usepackage{amsmath,amsthm}
\usepackage{amsfonts,amsmath,comment}
\usepackage{amssymb}
\usepackage{fullpage}
\usepackage{epsf}
\usepackage{color}
\usepackage{graphicx}
\usepackage{pdflscape}
\usepackage{rotating}
\usepackage{rotating}

\newtheorem{theorem}{Theorem}[section]
\newtheorem{lemma}[theorem]{Lemma}

\theoremstyle{definition}

\renewcommand{\geq}{\geqslant}
\renewcommand{\leq}{\leqslant}

\begin{document}
%
\title{
The existence of  square non-integer Heffter arrays
}

\author{
Nicholas J. Cavenagh\thanks{Department of Mathematics, The University of Waikato, Private Bag 3105, Hamilton 3240, New Zealand.}
\\
\texttt{nickc@waikato.ac.nz}
\and
Jeff Dinitz\thanks{Department of Mathematics and Statistics, University of Vermont, Burlington, VT 05405, USA}
\\
\texttt{jeff.dinitz@uvm.edu}
\and
Diane Donovan\thanks{School of Mathematics and Physics, The University of Queensland,  Queensland 4072, Australia.}
\\
\texttt{dmd@maths.uq.edu.au}
\and
\c{S}ule Yaz\i c\i\thanks{Department of Mathematics, Ko\c{c} University, Istanbul, Turkey}
\\
\texttt{eyazici@ku.edu.tr}
}

\maketitle

\begin{abstract}
A Heffter array $H(n;k)$ is an $n\times n$ matrix such that each row and column contains $k$ filled cells,
each row and column sum is divisible by $2nk+1$ and either $x$ or $-x$ appears in the array for each integer $1\leq x\leq nk$.
Heffter arrays are useful for embedding the graph $K_{2nk+1}$ on an orientable surface.
An integer Heffter array is one in which each row and column sum is $0$. 
Necessary and sufficient conditions (on $n$ and $k$) for the existence of an integer Heffter array $H(n;k)$ were verified by Archdeacon, Dinitz, Donovan and Yaz\i c\i  \ (2015) and Dinitz and Wanless (2017).
In this paper we consider square Heffter arrays that are not necessarily integer. We show that such 
Heffter arrays exist whenever $3\leq k<n$.
\end{abstract}


\section{Introduction}

A Heffter array $H(m,n;s,t)$ is an $m\times n$ matrix of integers such that:
\begin{itemize}
\item each row contains $s$ filled cells and each column contains $t$ filled cells;
\item the elements in every row and column sum to $0$ in ${\mathbb Z}_{2ms+1}$; and
\item for each integer $1\leq x\leq ms$, either $x$ or $-x$ appears in the array.
\end{itemize}

 If the
Heffter array is square, then $m = n$ and necessarily $s = t$. We denote such Heffter arrays by 
 $H(n; k)$, where each  row and each column contains $k$ filled cells.
A Heffter array is called an {\em integer} Heffter array if Condition 2 in the definition of a Heffter
array above is strengthened so that the elements in every row and every column sum to zero
in ${\mathbb Z}$.

 Archdeacon, in \cite{A}, was the first to define and study a Heffter array $H(m, n; s, t)$. He
showed that a Heffter array with a pair of special orderings can be used to construct
an embedding of the complete graph $K_{2ms+1}$ on a surface. This connection is formalised in the
following theorem. For definitions of simple and compatible orderings refer to \cite{A}.

\begin{theorem}\label{Archdeacon} {\rm \cite{A}} Given a Heffter array $H(m, n; s, t)$ with compatible orderings $\omega_r$ of the
symbols in the rows of the array and $\omega_c$ on the symbols in the columns of the array, then
there exists an embedding of $K_{2ms+1}$ such that every edge is on a face of size $s$ and a face of
size $t$. Moreover, if $\omega_r$ and $\omega_c$ are both simple, then all faces are simple cycles.
\end{theorem}

The embedding of $K_{2ms+1}$ given in Theorem \ref{Archdeacon} provides a connection with the embedding of
cycle systems. A $t$-cycle system on $n$ points is
a decomposition of the edges of $K_n$ into $t$-cycles. A $t$-cycle system $C$ on $K_n$ is cyclic if
there is a labeling of the vertex set of $K_n$ with the elements of $\Bbb{Z}_n$ such that the permutation
$x \rightarrow  x + 1$ preserves the cycles of $C$. A biembedding of an $s$-cycle system and a $t$-cycle
system is a face 2-colorable topological embedding of the complete graph $K_{2ms+1}$ in which
one color class is comprised of the cycles in the $s$-cycle system and the other class contains
the cycles in the $t$-cycle system, see for instance \cite{BDF,CMPP,DM,GG,GM,M,V} for further details.



A number of papers have appeared on the construction of Heffter arrays, $H(m,n;s,t)$. The case where the array contained no empty cells was studied in \cite{ABD}, with results summarised in Theorem \ref{Heffter}.

\begin{theorem}\label{Heffter} {\rm \cite{ABD}} There is an $H(m, n; n,m)$  for all $m, n \geq 3$ and an integer Heffter array $H(m, n; n,m)$ exists if and only if $m, n \geq 3$ and $mn \equiv 0, 3$
$(mod$ $4)$.
\end{theorem}

The papers \cite{ADDY,DW} focused on square integer Heffter arrays $H(n;k)$ and verified their existence for all admissible orders. This result is summarized
 in the following theorem.

\begin{theorem}\label{intergerHeffter}{\rm \cite{ADDY,DW}}
There exists an integer $H(n; k)$ if and only if $3 \leq k \leq  n$ and $nk \equiv
0, 3$ $(mod$ $4)$.
\end{theorem}

\begin{table}[h]\label{existence}
 \begin{center}
\begin{tabular} {|c||c|c|c|c|} \hline
$n\backslash k$&          0&                   1&          2&3            \\ \hline\hline
0              &\cite{ADDY}&\cite{ADDY,DW}&\cite{ADDY}&\cite{ADDY} \\ \hline
1              &\cite{ADDY}&                 DNE&        DNE&\cite{ADDY} \\ \hline
2              &\cite{ADDY}&                 DNE&\cite{ADDY}&         DNE \\ \hline
3              &\cite{ADDY}&\cite{ADDY,DW}&        DNE&          DNE \\ \hline
\end{tabular}
\end{center}
\caption{ Existence results for square integer Heffter arrays $H(n;k)$}
\end{table}

Table \ref{existence} lists the possible cases and cites the article which verifies existence of square integer Heffter arrays, where DNE represents  a value that does not exist. In these cases we will verify existence for the non-integer Heffter arrays $H(n;k)$.
The main result of this paper is the following. 
\begin{theorem}
There exists an $H(n; k)$ if and only if $3 \leq k \leq  n$. 
\label{mainthm}
\end{theorem}

From Theorem \ref{Heffter} above, the case $n=k$ has been solved, so we henceforth assume that $n>k$. 
The cases that need to be addressed are set out in Table 2. Cases A, B, C, D and E are solved by Theorems \ref{k=4t+6}, \ref{m=4m+3k=4p+3}, \ref{n=4m+2k=4p+3}, \ref{n=4m+1k=4p+5} and \ref{n=4m+2k=4p+5}, respectively, thus proving Theorem \ref{mainthm}.

\begin{table}[h]\label{thecases}
\begin{center}
\begin{tabular}{|l||l|l|l|l|l|l|}
\hline
   &Case A                  &Case B                  &Case C    &Case D  & Case E             \\
\hline\hline
$k$&$2$ (mod $4$)&$3$ (mod $4$) & $3$ (mod $4$)    & $1$ (mod $4$)  &$1$ (mod $4$)  \\
\hline
$n$&$1,3$ (mod $4$) & $3$ (mod $4$) & $2$ (mod $4$) & $1$ (mod $4$) & $2$ (mod $4$) \\
\hline
\end{tabular}
\caption{Cases for non-integer Heffter arrays $H(n;k)$}
\end{center}
\end{table}

In this paper the rows and columns of a square $n\times n$ array are always indexed by the elements of $\{1,2,\dots ,n\}$.
Unless otherwise stated, when working modulo $n$, replace $0$ by $n$, so we use the symbols $1,\dots, n$ instead of $0,\dots, n-1$.
While rows and columns are calculated modulo an integer, entries are always expressed as non-zero integers.
Throughout this paper $A[r,c] = x$ denotes the occurrence of symbol $x$ in cell $(r,c)$ of array $A$.

 By $A\pm z$ we refer to the array obtained by replacing $A[r,c]$ by $A[r,c]+z$ (if $A[r,c]>0$) and $A[r,c]-z$  (if  $A[r,c]<0$). 
If each row and each column of $A$ contains the same number
of positive and negative numbers, then $A\pm z$ has the same row and column sums as $A$. In this case we say   $A$  is {\em shiftable}. 
The {\em support} of an array $A$ is defined to be the set containing the absolute value
of the elements contained in $A$.
If $A$ is an array with support $S$ and $z$ a nonnegative
integer, then  $A\pm z$ has support $S + z$.

\section{Increasing $k$ from base cases}

For each of the cases set out in Table 2 our overall strategy is to generate a base case  $H(n;k)$ where $k$ takes the smallest possible value and then increase $k$ by multiples of $4$, adjoining 4 additional entries to each row and column. In this section we outline various tools to enable this process. To this end, we introduce the following definitions. 

We associate the cells of an $n\times n$ array with
the complete bipartite graph $K_{n,n}$ where partite sets
are denoted $\{a_i\mid i=1,\dots ,n\}$ and $\{b_j\mid j=1\dots n\}$
and the edge $\{a_i,b_j\}$ corresponds to the cell $(i,j)$.
We say that in an $n\times n$ array a set of cells $S$ forms a $2$-{\em factor} if the corresponding set of edges in the graph $K_{n,n}$ forms a spanning $2$-regular graph  and  forms a {\em Hamilton cycle}
if the corresponding set of edges forms a single cycle of length $2n$.

For each $d\in \{0,1,\dots, n-1\}$, we define the {\em diagonal} $D_d$ to be the set of cells of the form $(r+d,r)$, $1\leq r\leq n$ (evaluated modulo $n$).
Observe that the cells $D_i\cup D_j$ form a Hamilton cycle whenever $j-i$ is coprime to $n$.

\begin{lemma}
Let $S_1$ and $S_2$ be two disjoint sets of cells in an $n\times n$ array which each form Hamilton cycles.
The cells of $S_1\cup S_2$ can be filled with the elements of $\{1,2,\dots ,4n\}$ so that each row and column sum is equal to  $8n+2$.
\label{bacon}
\end{lemma}

\begin{proof}
Let the cells of $S_1$ and $S_2$ be $\{e_i\mid 1\leq i\leq 2n\}$ and $\{f_i\mid 1\leq i\leq 2n\}$, respectively, where: 
\begin{itemize}
\item Cells $e_{i}$ and $e_{i+1}$ are in the same row (column) whenever $i$ is odd (respectively, even);
\item Cells $f_{i}$ and $f_{i+1}$ are in the same row (column) whenever $i$ is odd (respectively, even);
\item  Cells $e_1$ and $f_1$ are in the same row.
\end{itemize}
Place $1$ in cell $e_1$, $4n$ in cell $f_1$ and: 
\begin{itemize}
\item $2n-2i+1$ in cell $e_{2i+1}$, where $1\leq i\leq n-1$; \quad $2n+2i-1$ in cell $e_{2i}$ where $1\leq i\leq n$;
\item $2n+2i$ in cell $f_{2i+1}$ where $1\leq i\leq n-1$; \quad $2n-2i+2$ in cell $f_{2i}$ where $1\leq i\leq n$.
\end{itemize}
The entries in cells $e_1$, $e_2$, $f_1$ and $f_2$ add to $1+(2n+1)+4n+2n=8n+2$.
For every other row, there are two cells from $S_1$ with entries adding to $4n+2$ and two cells from $S_2$ with entries adding to
$4n$. For every column, there are two cells from $S_1$ adding to $4n$ and two cells from $S_2$ adding to $4n+2$.
See the example below.
\end{proof}

We demonstrate Lemma \ref{bacon} below when $n=9$. The elements of $S_1$ are shown in bold.

\begin{center}
{\small $
\renewcommand{\arraycolsep}{2pt}
\begin{array}{|c|c|c|c|c|c|c|c|c|}
\multicolumn{9}{c}{S_1\cup S_2} \\
\hline
{\bf 1} & {\bf 19} & 18 & & & & 36 & & \\
\hline
34 & {\bf 17} & {\bf 21} & & & & 2 & & \\
\hline
& & {\bf 15} & {\bf 23} & 10 & 26 & & & \\
\hline
4 & & & {\bf 13} & {\bf 25} & & & & 32 \\
\hline
& 14 & & & {\bf 11} & {\bf 27} & & 22 & \\
\hline
& & 20 & & & {\bf 9} & {\bf 29} & 16 & \\
\hline
& & & 30 & & & {\bf 7} & {\bf 31} & 6 \\
\hline
& 24 & & & & 12 & & {\bf 5} & {\bf 33} \\
\hline
{\bf 35} & & & 8 & 28 & & & & {\bf 3} \\
\hline
\end{array}
$}
\end{center}

The following theorem will be crucial in Cases A and D.

\begin{theorem}
Let $H(n;k)$ be a Heffter array such that each row and column sums to $2nk+1$.
Suppose there exist Hamilton cycles $H_1$ and $H_2$  disjoint to each other and to the filled cells of $H(n;k)$.
Then there exists an $H(n;k+4)$ Heffter array with row and column sums equal to $2n(k+4)+1$,
where the filled cells are precisely the filled cells of $H(n;k)$, $H_1$ and $H_2$.
\label{crucial}
\end{theorem}

\begin{proof}
Let $A_0$ represent the $H(n;k)$ and negate each element so that each row and column has sum equal to $-(2nk+1)$.
From Lemma \ref{bacon}, there exists an array $A_1'$ on the cells of $H_1$ and $H_2$ such that each row and column sum is equal to $8n+2$; 
add $n(k+4)-(4n)=nk$ to each element of $A_1'$ to create a new array $A_1$ that has support 
$\{nk+1,nk+2,\dots ,n(k+4)\}$. Note that in $A_1$ each row and column sum is equal to $8n+2+4nk$. Let $A$ be the union of $A_0$ with $A_1$.
The resulting array $A$ has support $\{1,2,\dots ,n(k+4)\}$, with $k+4$ filled cells in each row and column. Finally, each row and column sum of $A$ is 
$-(2nk+1)+(8n+2)+4(nk)=2n(k+4)+1$, as desired.
\end{proof}

The following lemma generalizes Theorem 2.2 from \cite{DW}, and is used in Cases B and C. 

\begin{lemma}
Let $S_1$ and $S_2$ be two disjoint sets of cells in an $n\times n$ array which each form Hamilton cycles. Then for any positive integers $t$
 and $s>t+2n$,
the cells of $S_1\cup S_2$ can be filled with elements to make a shiftable array with support   
$\{s+i,t+i\mid 1\leq i\leq 2n\}$ so that the four elements in each row and each column sum to $0$.
\label{hamilton}
\end{lemma}

\begin{proof}
Let the sets of cells of $S_1$ and $S_2$ be $\{e_i\mid 1\leq i\leq 2n\}$ and $\{f_i\mid 1\leq i\leq 2n\}$, respectively,
where:
\begin{itemize}
\item Cells $e_{i}$ and $e_{i+1}$ are in the same row (column) whenever $i$ is odd (respectively, even);
\item Cells $f_{i}$ and $f_{i+1}$ are in the same row (column) whenever $i$ is odd (respectively, even);
\item  Cells $e_1$  and $f_1$ are in the same row.
\end{itemize}
Place:
\begin{itemize}
\item $s+2n$ in cell $e_1$ and $-(t+2n)$ in cell $f_1$, with sum $s-t$;
\item $s+2i$ in cell $e_{2i+1}$  and $-(t+2i)$ in cell $f_{2i+1}$, with sum $s-t$, for  $1\leq i\leq n-1$,
\item $-(s+2i-1)$ in cell $e_{2i}$ and  $t+2i-1$ in cell $f_{2i}$, with sum $t-s$, for  $1\leq i\leq n$.
\end{itemize}
It now follows that the row sums are $0$. Using similar arguments it can be seen that the columns also sum to  $0$.
Observe that there are two positive and two negative integers in each row and column; thus the array is shiftable. 
\end{proof}

The proof of the following lemma is similar to the proof of Lemma \ref{hamilton}; we use this in Case E.

\begin{lemma}
Let $n$ be even.
Let $S_1$ and $S_2$ be two disjoint sets of cells in an $n\times n$ array which each form $2$-factors that are the union of two $n$-cycles. Then for any positive integers $s$, $t$, $u$ and $v$ where $s>t+n$, $t>u+n$ and $u>v+n$, 
the cells of $S_1\cup S_2$ can be filled with elements to make a shiftable array with support
$\{s+i,t+i,u+i,v+i\mid 1\leq i\leq n\}$ so that the four elements in each row and column sum to $0$.
\label{twofactor}
\end{lemma}

\begin{proof}
Let $C_i$, $C_i'$ be the cycles of the $2$-factor $S_i$, $i\in \{1,2\}$, where $C_1$ and $C_1'$ share a row and $C_2$ and $C_2'$ share a row. 
Let the sets of cells of $C_1$, $C_1'$, $C_2$ and $C_2'$ be
$\{e_i\mid 1\leq i\leq n\}$, $\{f_i\mid 1\leq i\leq n\}$, $\{g_i\mid 1\leq i\leq n\}$ and $\{h_i\mid 1\leq i\leq n\}$, respectively,
where:
\begin{itemize}
\item Cells $e_{i}$ and $e_{i+1}$ are in the same row (column) whenever $i$ is odd (respectively, even);
\item Cells $f_{i}$ and $f_{i+1}$ are in the same row (column) whenever $i$ is odd (respectively, even);
\item Cells $g_{i}$ and $g_{i+1}$ are in the same row (column) whenever $i$ is odd (respectively, even);
\item Cells $h_{i}$ and $h_{i+1}$ are in the same row (column) whenever $i$ is odd (respectively, even);
\item  Cells $e_1$ and $g_1$ are in the same row; Cells $f_1$ and $h_1$ are in the same row.
\end{itemize}
Place:
\begin{itemize}
\item $s+n$ in cell $e_1$, $-(t+n)$ in cell $g_1$ and $u+n$ in cell $f_1$, $-(v+n)$ in cell $h_1$;
\item $s+2i$ in cell $e_{2i+1}$ and $-(t+2i)$ in cell $g_{2i+1}$, for $1\leq i\leq n/2-1$;
\item $-(s+2i-1)$ in cell $e_{2i}$ and $t+2i-1$ in cell $g_{2i}$, for  $1\leq i\leq n/2$;
\item $u+2i$ in cell $f_{2i+1}$ and $-(v+2i)$ in cell $h_{2i+1}$, for  $1\leq i\leq n/2-1$;
\item $-(u+2i-1)$ in cell $f_{2i}$ and $v+2i-1$ in cell $h_{2i}$, for $1\leq i\leq n/2$.
\end{itemize}

It now follows that the row sums are $0$. Using similar arguments it can be seen that the columns also sum to  $0$.
\end{proof}

\section{Case A: $k\equiv 2$ (mod $4$)}

In this section we construct a Heffter array $H(n;k)$, for $n\equiv 1,3 \ (\mbox{mod }4)$ and $k\equiv 2$ (mod $4$), where $k<n$.
Row and column sums will always equal $2nk+1$. We start with an example of our construction. 

\begin{center}
{\small $
\renewcommand{\arraycolsep}{2pt}
\begin{array}{|c|c|c|c|c|c|c|c|c|c|c|c|c|c|c|}
\multicolumn{15}{c}{H(15;6)} \\ 
\hline
  6&   &   &   &   &   &   &   &   &- 4& 89& 81&  1&  8 &\\
\hline
   & 12&   &   &   &   &   &   &   &   &-88& 83& 87& 85&  2\\
\hline
 86&   & 18&   &   &   &   &   &   &   &   &-82& 77&  3& 79\\
\hline
 73& 80&   & 24&   &   &   &   &   &   &   &   &-76& 71&  9\\
\hline
 15& 67& 74&   & 30&   &   &   &   &   &   &   &   &-70& 65\\
\hline
 59& 21& 61& 68&   & 36&   &   &   &   &   &   &   &   &-64\\
\hline
-58& 53& 27& 55& 62&   & 42&   &   &   &   &   &   &   &   \\
\hline
   &-52& 47& 33& 49& 56&   & 48&   &   &   &   &   &   &   \\
\hline
   &   &-46& 41& 39& 43& 50&   & 54&   &   &   &   &   &   \\
\hline
   &   &   &-40& 35& 45& 37& 44&   & 60&   &   &   &   &   \\
\hline
   &   &   &   &-34& 29& 51& 31& 38&   & 66&   &   &   &   \\
\hline
   &   &   &   &   &-28& 23& 57& 25& 32&   & 72&   &   &   \\
\hline
   &   &   &   &   &   &-22& 17& 63& 19& 26&   & 78&   &   \\
\hline
   &   &   &   &   &   &   &-16& 11& 69& 13& 20&   & 84&   \\
\hline
   &   &   &   &   &   &   &   &-10&  5& 75&  7& 14&   & 90\\
\hline
   \end{array}
   $
}
\end{center}

\medskip

\begin{lemma}\label{k=6} For $n\equiv\ 1,3\ (\mbox{mod }4)$, $n\geq 7$ and $k=6$ there exists a Heffter array $H(n;6)$.
\end{lemma}

\begin{proof}
We remind the reader that rows and columns are calculated modulo $n$ but the array entries are not.
The array $A=A[r,c]$ is defined as follows, where $1\leq i\leq n$:
\begin{center}
$\begin{array}{lll}
A[i,i]=6i, & A[i+2,i]=6n+2-6i, \\
 A[i+1,n-2+i]=6n+1-6i, &
A[i+2,n-2+i]=6i-3, \\
 A[i,n-5+i]=6n+5-6i, & A[i+1,n-5+i]=-6n-4+6i. 
\end{array}$
\end{center}

Then the support of $A$ is $\{1,\dots,6n\}$.
The sets of elements in rows $1$, $2$ and $i$, $3\leq i\leq n$, are, respectively:
\begin{center}
$\begin{array}{l}
\{6,8,1,6n-9,6n-1,-4\}, \qquad \{12,2,6n-5,6n-3,6n-7,-(6n-2)\}, \\
\{6i,6n+2-6(i-2),6n+1-6(i-1),6(i-2)-3,6n+5-6i, -6n-4+6(i-1)\}.
\end{array}$
\end{center}
Thus in each case the sum of elements in a row is $12n+1$. 

The set of elements in column $i$, $1\leq i\leq n-5$ is:
$$\{6i,6n+2-6i,6n+1-6(i+2), 6(i+2)-3,6n+5-6(i+5), -6n-4+6(i+5)\}.$$
The set of elements in columns $n-4$, $n-3$, $n-2$, $n-1$ and $n$ are, respectively:
\begin{center}
$\begin{array}{ll}
\{6n-24,26,13,6n-15,6n-1,-(6n-2)\}, & 
\{6n-18,20,7,6n-9,6n-7,-(6n-8)\}, \\
\{6n-12,14,1,6n-3,6n-13,-(6n-14)\}, &
\{6n-6,8,6n-5,3,6n-19,-(6n-20)\}, \\
\{6n,2,6n-11,9,6n-25,-(6n-26)\}.
\end{array}$
\end{center}
Thus in each case the sum of elements in a column is $12n+1$. 
\end{proof}


\begin{theorem}\label{k=4t+6} There exists a Heffter array $H(n;k)$ for all $n\equiv 1,3\ (\mbox{mod }4)$ and $k\equiv 2\ (\mbox{mod }4)$, where $n>k\geq 6$.
\end{theorem}

\begin{proof}
Let $k=4p+6$. Then $4p+6\leq n-1$ so $p\leq (n-7)/4$. We have solved the case $p=0$ in Lemma \ref{k=6} so we may assume $p\geq 1$.  
Observe that the Hefter array given in the proof of that lemma uses only elements in diagonals $D_0,D_2,D_3,D_4,D_5$  and $D_6$ and so does not intersect the diagonals $D_7,D_8,\dots , D_{n-1}$. 
We can  apply Theorem \ref{crucial} recursively, where the diagonals $D_7,D_8,\dots ,$ $D_{6+2p-1},D_{6+2p}$ can be paired to give sets of cells $S_1$ and $D_{6+2p+1},D_{6+2p+2},\dots , D_{6+4p-1},D_{6+4p}$ paired to give sets of cells $S_2$.
The result is a Heffter array $H(n;k)$ with constant row and column sum $2nk+1$ whenever $k\equiv 2$ (mod $4$), $n\equiv 1,3$ (mod $4$) and $n>k\geq 6$.
\end{proof}

\section{Case B: $k\equiv\ 3$ (mod $4$) and $n\equiv\ 3$  (mod $4$)}

In this section we construct a Heffter array $H(n;k)$ where $n=4m+3$ and $k=4p+3$, where $k<n$.  We first assume that $m\geq 4$.

 We will begin with  $k=3$ and  construct $n\times n$ array which is the concatenation of three smaller arrays, $A_0=A_0[r,c]$, of dimension $(4m-7)\times (4m-7)$, $A_1=A_1[r,c]$ of dimension $7\times 7$ and $C$ of dimension $3\times 3$, each containing $3$ filled cells per row and column. So we see that $n= 4m+3$.  The sum of the rows and columns in $A_0$ and $A_1$ will be $0$, while the sum of the  rows and columns in $C$ will be $2nk+1$.

We begin with an example of the main construction of this section.

\begin{center}
{\small $
\renewcommand{\arraycolsep}{2pt}
\begin{array}{|c|c|c|c|c|c|c|c|c|c|c|c|c|c|c|c|c|c|c|c|c|}
\multicolumn{19}{c}{H(19;3)} \\
\hline
16&-48&32&&&&&&&&&&&&&&&&\\
\hline
17&27&&-44&&&&&&&&&&&&&&&\\
\hline
-33&&-14&&47&&&&&&&&&&&&&&\\
\hline
&21&&15&&-36&&&&&&&&&&&&&\\
\hline
&&-18&&-13&&31&&&&&&&&&&&&\\
\hline
&&&29&&-9&&-20&&&&&&&&&&&\\
\hline
&&&&-34&&-12&&46&&&&&&&&&&\\
\hline
&&&&&45&&-10&-35&&&&&&&&&&\\
\hline
&&&&&&-19&30&-11&&&&&&&&&&\\
\hline
&&&&&&&&&-25	&24	      &1	   &        &        &         &         &          &         &         \\		
\hline	
&&&&&&&&&22    &-50 &		   &28    &        &         &         &          &         &         \\	

\hline		
&&&&&&&&&3		&         &37   &	    &-40    &         &         &          &         &         \\	

\hline 		
&&&&&&&&&	    &26	  &        &23    &	     &-49 &	     &          &         &         \\

\hline
&&&&&&&&&        &         &-38&        &42   &		   &-4       &          &         &         \\

\hline
&&&&&&&&&		&         &        &-51&        &8       &   43 &          &         &         \\

\hline
&&&&&&&&&		&         &        &        &-2      &	41  &-39 &          &         &         \\

\hline
&&&&&&&&&        &         &        &        &        &         &         &  5       &57    &53    \\

\hline
&&&&&&&&&        &         &        &        &        &         &         &  54   &6        &55   \\

\hline
&&&&&&&&&        &         &        &        &        &         &         &  56   &52   &7        \\
        \hline
        \end{array}$
}
        \end{center}

\begin{lemma}\label{n=4m+3k=3} For $n\equiv\ 3\ (\mbox{mod }4)$, $n\geq 7$ and $k=3$ there exists a Heffter array $H(n;3)$.
\end{lemma}

\begin{proof}
Let $n=4m+3$. 
We first assume that $m\geq 4$ and so $n\geq 19$. The small cases will be dealt with at the end of the proof. 
Let $A_0'$ be: 
$$\begin{array}{lll}
A_0'[2i-1,2i] = 8m+1-i,       \qquad A_0'[2i,2i-1] = -(8m+i),      &      & 1\leq i\leq m,\\
A_0'[2i,2i+1] = 12m-i,    \ \ \  \     \qquad A_0'[2i+1,2i]=-(4m+1+i),         &   & 1\leq i\leq m-1,\\
A_0'[2m-2+2i,2m-1+2i] = 5m+i, &  & 1 \leq i\leq m-3,\\
A_0'[2m-1+2i,2m-2+2i] = -(11m+1-i), &  & 1 \leq i\leq m-3,\\
A_0'[2m-1+2i,2m+2i] = 9m+i,   &    & 1 \leq i\leq m-4,\\   
 A_0'[2m+2i,2m-1+2i] = -(7m+1-i), &   & 1 \leq i\leq m-4,\\
A_0'[i+1,i+1] = -(4m-1-i),        &                               & 1\leq i\leq 2m-2, \\
A_0'[2m+i,2m+i] = 2m-i,        &                                  & 1 \leq i\leq 2m-8, \\
\quad A_0'[2m,2m] = 4m-1,            \quad A_0'[1,4m-7] = -12m, \quad A_0'[1,1] = 4m, \\
 \quad A_0'[4m-7,1] = 4m+1,  
\quad A_0'[4m-7,4m-7]=6m+3.
\end{array}$$
We illustrate $A_0'$ in the case $m=4$:
\begin{center}
{\small $
\renewcommand{\arraycolsep}{2pt}
\begin{array}{|c|c|c|c|c|c|c|c|c|}
\multicolumn{9}{c}{A_0' (m=4)} \\
\hline
16 & 32 & & & & & & & -48 \\
\hline
-33 & -14 & 47 & & & & & &  \\
\hline
 & -18 & -13 & 31 & & & & &  \\
\hline
 &  & -34 & -12 & 46 & & & &  \\
\hline
 &  & & -19 & -11 & 30 & & &  \\
\hline
 &  & & & -35 & -10 & 45 & &  \\
\hline
 &  & &  & &  -20 & -9 & 29 &  \\
\hline
 &  & &  &  & & -36 & 15 & 21   \\
\hline
17 &  & &  &  &  &  & -44 & 27 \\
\hline
\end{array}$}
\end{center}

First observe that the array $A_0'$ is a $(4m-7)\times (4m-7)$ array that has $3$ filled cells in each row and column. 

To confirm that the row and columns sums are $0$, 
note that this array was constructed by taking the first $4m-8$ rows and columns  of the integer Heffter array $H(4m;3)$ given in \cite{ADDY}, then placing entry $-12m$ in cell $(1,4m-7)$, entry $4m+1$ in cell $(4m-7,1)$ and entry  $6m+3$ in cell $(4m-7,4m-7)$.  Thus we need only check  the sum of row $4m-7$ which is $(4m+1)+(6m+3)-(10m+4)=0$ and the sum of column $4m-7$ which is $-12m+(6m-3)+(6m+3)=0$. Hence each row and column in the array $A_0'$ sums to zero.

Although not necessary for the case $k=3$, for larger values of $k$ (see the following theorem) we map the rows and columns of $A_0'$ so that the filled cells are a subset of the union of diagonals $D_0\cup D_1\cup D_2\cup D_{n-2}\cup D_{n-1}$. This can be done by applying the mapping
\begin{eqnarray*}
i&\rightarrow &
\left\{
\begin{array}{ll}
 2i-1,&\mbox{ when }1\leq i\leq 2m-3,\\
8m-2i-12,&\mbox{when }2m-2\leq i\leq 4m-7,
\end{array}\right.
\end{eqnarray*}
 to the rows and column of $A_0'$. This does not change the row and column sum and the support is still  $\{1\dots,4m+3\}\setminus \{1,2,\dots,7,2m,6m-2,\dots,6m+2, 6m+4,10m-3,\dots, 10m+3,12m+1,12m+2,\dots,12m+9\}.$ 
We call this rearranged array $A_0$; see $H(19;3)$ above for $A_0$ when $m=4$.

\renewcommand{\tabcolsep}{3pt}
Next, let $A_1$ be:
\begin{center}
{\small $
\renewcommand{\arraycolsep}{2pt}
\begin{array}{|c|c|c|c|c|c|c|}
\multicolumn{7}{c}{A_1}\\
\hline
-(6m+1)	&6m	      &1	   &        &        &         &         \\		
\hline	
6m-2    &-(12m+2) &		   &6m+4    &        &         &         \\	
\hline		
3		&         &10m-3  &	    &-10m    &         &         \\	
\hline 		
	    &6m+2	  &        &6m-1    &	     &-(12m+1) &	     \\
\hline
        &         &-(10m-2)&        &10m+2   &		   &-4       \\
\hline
		&         &        &-(12m+3)&        &2m       &   10m+3 \\
\hline
		&         &        &        &-2      &	10m+1  &-(10m-1) \\
\hline
        \end{array}
$}
     \end{center}

It is easy to check that this array has row and column sum $0$ and support $\{1,2,3,4,2m,6m-2,\dots,6m+2, 6m+4,10m-3,\dots, 10m+3,12m+1,12m+2,12m+3\}$.
We place $A_1$ on the intersection of row and column sets $\{4m-6,4m-5,\dots ,4m\}$.

Finally, place the block $C$ on the intersection of the row and column sets $\{4m+1,4m+2,4m+3\}$.
Observe that the rows and columns sum to $2nk+1$. It is convenient to express $C$ in terms of $n$ and $k$, as it will be part of more general constructions in the next theorems. However, if $k=3$, observe that $\{12m+4,12m+5,\dots ,12m+9\}=\{nk-5,nk-4,\dots ,nk\}$.

\begin{center}
{\small $
\renewcommand{\arraycolsep}{2pt}
\begin{array}{|c|c|c|}
\multicolumn{3}{c}{C} \\
\hline
5          & nk    & nk-4 \\
\hline
nk-3       &  6    & nk-2 \\
\hline
nk-1       & nk-5  &   7\\
\hline
\end{array}
$}
\end{center}

Let $A$ be the concatenation of $A_0$, $A_1$ and $C$ to obtain an $H(4m+3;3)$ Heffter array for all $m\geq 4$. An $H(7;3)$ is given in the Appendix. When $n= 11$ or 15, concatenating the array $B(n)$ given in the Appendix for these size with $C$ will produce an $H(n;3)$ with similar properties and support.
\end{proof}

We now consider the case when $k>3$; we apply the techniques developed in Section 2. 

\begin{theorem}\label{m=4m+3k=4p+3} There exists a Heffter array $H(n;k)$ for all $n\equiv 3(\mbox{mod }4)$ and $k\equiv 3\ (\mbox{mod }4)$, where $n>k\geq 3$.
\end{theorem}

\begin{proof} Given Lemma \ref{n=4m+3k=3} we only need to address the case $k=4p+3$ where $1\leq p<m$.
Since $k\leq n-4$, $p\leq (n-7)/4$.  
First observe that the filled cells of $H(n;3)$ given in the previous lemma are a subset the union of diagonals $D_0, D_1,D_2, D_{n-2}, D_{n-1}$ with support $\{1,2,\dots ,12m+3\}\cup\{nk-5,\dots ,nk\}$. 
We next identify $p$ disjoint Hamilton cycles, that are also disjoint from diagonals $D_0,D_1,D_2, $ $D_{n-2},D_{n-1}$,
 by pairing the remaining $(n-5)$ diagonals. 
Then  we apply Lemma \ref{hamilton}, contributing $\{12m+4,\dots ,nk-6\}$ to the support. 
The result is a Heffter array $H(n;k)$ with each row and column sum equal to $0$ except for the final three rows and columns which sum to $2nk+1$.
\end{proof}

\section{Case C: $k\equiv 3$ (mod $4$) and $n\equiv 2$ (mod $4$)}

We construct a Heffter array $H(n;k)$ where $n=4m+6$ and $k=4p+3$ for $k<n$. 
First we  assume that $m\geq 7$.

As in the previous section,  we will begin with  $k=3$ and  construct $n\times n$ array which is the concatenation of several smaller arrays.  Here we use four smaller arrays, $A_0=A_0[r,c]$ of size $(4m-13)\times (4m-13)$, $A_1=A_1[r,c]$ of size $13\times 13$, $A_2=A_2[r,c]$ of size $3\times 3$ and $A_3=A_3[r,c]$ of size $3\times 3$, each containing $3$ filled cells per row and column. The sum of the rows and columns in $A_0$, $A_1$ and $A_2$ is $0$, while the sum of the  rows and columns in $A_3$ is $2nk+1$.

\begin{lemma}\label{n=4m+2k=3} For $n\equiv 2\ (\mbox{mod }4)$, $n\geq 34$ and $k=3$ there exists a Heffter array $H(n;3)$.
\end{lemma}

\begin{proof}
Let $n=4m+6$ and $m\geq 7$. 
Let $A_0'$ be:
$$
\renewcommand{\arraycolsep}{3pt}
\begin{array}{lll}
A_0'[2i-1,2i] = 8m+1-i,  \qquad      A_0'[2i,2i-1] = -(8m+i),      &      & 1\leq i\leq m,\\
A_0'[2i,2i+1] = 12m-i,  \quad \qquad \   A_0'[2i+1,2i]=-(4m+1+i),      &      & 1\leq i\leq m-1,\\
A_0'[2m-2+2i,2m-1+2i] = 5m+i, &   & 1\leq i\leq m-6,\\
 A_0'[2m-1+2i,2m-2+2i] = -(11m+1-i), & & 1\leq i\leq m-6,\\
A_0'[2m-1+2i,2m+2i] = 9m+i,   &    & 1\leq i\leq m-7,\\
 A_0'[2m+2i,2m-1+2i] = -(7m+1-i), &   & 1\leq i\leq m-7,\\
A_0'[i+1,i+1] = -(4m-1-i),        &                               & 1\leq i\leq 2m-2, \\
A_0'[2m+i,2m+i] = 2m-i,        &                                  & 1\leq i\leq 2m-14, \\
\quad A_0'[2m,2m] = 4m-1,   \qquad        A_0'[1,4m-13] = -(12m),&  & A_0'[1,1] = 4m,  \\
     A_0'[4m-13,1] = 4m+1, 
\quad A_0'[4m-13,4m-13]=6m+6. & &
\end{array}
$$

We illustrate $A_0'$ when $m=7$:
\begin{center}
{\small $
\renewcommand{\arraycolsep}{.1 pt}
\begin{array}{|c|c|c|c|c|c|c|c|c|c|c|c|c|c|c|}
\multicolumn{15}{c}{A_0' (m=7)}\\
\hline
28	   &56   & &        &        &        &        &        &        &        &        &        & & &-84 \\ \hline	
-57   & -26      &83	     & &        &        &        &        &        &        &        &       & & &\\ \hline				
	   &-30&-25  & 55       &        &      &	     &        &        &        &        &  &     & &\\ \hline
       &        &     -58       & -24       &82  && &       &        &        &        &        & &&\\ \hline
	   &        &        &-31  &-23 &  54      &        & &        &        &        &       & &&\\ \hline
	   &        &        &        &-59 &  -22      &  81      & &  &        &        &       & &&\\ \hline
       &        &        &        &        &-32&-21      & 53       &        &     &        &       & &&\\ \hline
       &        &        &        &        &        &-60  & -20       &  80      & &      &        &&&\\ \hline
       &        &        &        &        &        &        &-33  &-19     &  52      &        & &&&\\ \hline
       &        &        &        &        &        &        &        &-61 &  -18      &  79      & &&&\\ \hline
       &        &        &        &        &        &        &        &        &-34       &-17   & 51  &    & & \\ \hline
       &        &        &        &        &        &        &        &        &        &-62&-16   &78  & &\\ \hline
   &        &        &        &        &        &        &        &        &        & &-35   &-15  & 50 &\\ \hline      
 &        &        &        &        &        &        &        &        &        & &   &-63  & 27 &  36 \\ \hline    
29 &        &        &        &        &        &        &        &        &        & &  &  & -77 & 48 \\ \hline          
\end{array}$ }
\end{center}

Observe that the array $A_0'$ is a $(4m-13)\times (4m-13)$ array that has $3$ filled cells in each row and column.
Similarly to the previous case, this array was constructed by taking the first $4m-14$ rows and columns  of the integer Heffter array $H(4m;3)$ given in \cite{ADDY}, then placing entry $-12m$ in cell $(1,4m-13)$, entry $4m+1$ in cell $(4m-13,1)$ and entry  $6m+6$ in cell $(4m-13,4m-13)$.  Thus we need only check  the sum of row $4m-13$ which is $(4m+1)+(6m+6)-(10m+7)=0$ and the sum of column $4m-13$ which is $-12m+(6m-6)+(6m+6)=0$. Hence all rows and columns in the array $A_0'$ sum to zero. We apply the same mapping as in the proof of Lemma \ref{n=4m+3k=3} to the rows and columns so that all non-empty cells  are a subset of the diagonals $D_0, D_1, D_2, D_{n-2}, D_{n-1}$. Let $A_0$ be the resultant array. The support for $A_0$ is $\{1,\dots, 12m+8\}\setminus \{1,\dots,13,2m,6m-5,\dots,6m+5,6m+7,10m-6,\dots,10m+6,12m+1,12m+2,\dots,12m+18\}$.

 We now arrange the $57$ missing symbols into a $13\times13$ array $A_1$ and two  $3\times 3$ arrays $A_2$  and $A_3$.
Here is $A_1$:
\begin{center}
{\tiny $
\renewcommand{\arraycolsep}{.1 pt}
\begin{array}{|c|c|c|c|c|c|c|c|c|c|c|c|c|c|c|}
\multicolumn{13}{c}{A_1}\\
\hline
1	   &10m-4   &-(10m-3)&        &        &        &        &        &        &        &        &        &\\ \hline	
6m-5   & 2      &	     &-(6m-3) &        &        &        &        &        &        &        &        &\\ \hline		
-(6m-4)&        &        &-(6m+7)&12m+3   &        &        &        &        &        &        &        &\\ \hline 		
	   &-(10m-2)&10m+4   &        &        &-6      &	     &        &        &        &        &        &\\ \hline
       &        &-7      &        &        &10m+1   &-(10m-6)&        &        &        &        &        &\\ \hline
	   &        &        &12m+4   &-(6m+5) &        &        &-(6m-1) &        &        &        &        &\\ \hline
	   &        &        &        &-(6m-2) &        &        &-(6m+3) &12m+1   &        &        &        &\\ \hline
       &        &        &        &        &-(10m-5)&-5      &        &        &10m     &        &        &\\ \hline
       &        &        &        &        &        &10m-1   &        &        &-(10m+3)&4       &        &\\ \hline
       &        &        &        &        &        &        &12m+2   &-6m     &        &        &-(6m+2) &\\ \hline
       &        &        &        &        &        &        &        &-(6m+1) &        &        &-(6m+4) &12m+5\\ \hline
       &        &        &        &        &        &        &        &        &3       &10m+2   &        &-(10m+5)\\ \hline
       &        &        &        &        &        &        &        &        &        &-(10m+6)&12m+6   &-2m\\ \hline
        \end{array}$ }
\end{center}

\enlargethispage{3\baselineskip}


The support for $A_1$ is $\{1,\dots,7,2m,6m-5,\dots,6m+5,6m+7,10m-6,\dots,10m+6,12m+1,12m+2,\dots,12m+6\}$.
Finally we give $A_2$ and $A_3$ with support $\{8,9,10,11,12,13\}\cup 
\{nk-11,\dots ,nk\}$. In the case $k=3$, observe that $\{nk-11,\dots ,nk\} =\{12m+7,\dots,12m+18\}$. 
\begin{center}
{\small $
\renewcommand{\arraycolsep}{2pt}
\begin{array}{|c|c|c|}
\multicolumn{3}{c}{A_2} \\
\hline
-8	            &nk-2   &-(nk-10) \\ \hline	
-(nk-9)&-9               &nk \\ \hline		
nk-1  &-(nk-11) &-10               \\ \hline 	 	
\end{array}\quad \quad \quad
\begin{array}{|c|c|c|}
\multicolumn{3}{c}{A_3} \\
\hline
11                   & nk-3 & nk-7\\
\hline
nk-6       &             12 & nk-5\\
\hline
nk-4       & nk-8 &             13\\
\hline
\end{array}$}
\end{center}
\end{proof}


\begin{theorem}\label{n=4m+2k=4p+3} There exists a Heffter array $H(n;k)$ for all $n\equiv 2(\mbox{mod }4)$ and $k\equiv 3\ (\mbox{mod }4)$, where $n>k\geq 3$.
\end{theorem}
\begin{proof}
An $H(6;3)$ is given in the Appendix. Otherwise let $n=4m+6$ where $m\geq 1$. When $1\leq m\leq 6$ concatenate the $B(4m+6)$ given in the Appendix with the array $C$ given in Case B to get an $H(4m+6;3)$. Observe that when $1\leq m\leq 4$ the entries are only on diagonals $D_0, D_1, D_2, D_{n-2}, D_{n-1}$ as before. When $n=26$ the entries are on diagonals $D_0, D_1, D_2, D_{24}, D_{25}$ and $D_{9}$. We pair up the rest of the diagonals as $\{D_{2i},D_{2i+1}\}$ for $5\leq i\leq 11$, and  $\{D_{3},D_{4}\}$,  $\{D_{5},D_{6}\}$, $\{D_{7},D_{8}\}$ to get the required Hamilton cycles. When $n=30$, an $H(30;3)$ is given in the Appendix. Two Hamilton cycles $H$ and $K$ are also given as a reference on the array. These Hamilton cycles together with $H(30;3)$ only have entries on diagonals $D_0, D_1, D_2, D_3, D_{27}, D_{28}, D_{29}, D_{18}$ and $D_{11}$. For the rest of the diagonals,  pair them up as  $\{D_4,D_5\}$, $\{D_6,D_7\}$, $\{D_8,D_9\}$, $\{D_{12},D_{13}\}$, $\{D_{14},D_{15}\}$, $\{D_{16},D_{17}\}$ and $\{D_{2i-1},D_{2i}\}$ for $10\leq i\leq 13$ to get the necessary Hamilton cycles.

When $m\geq7$, concatenation of four smaller arrays, $A_0$, $A_1$, $A_2$ and $A_3$ gives the required $H(n;3)$ and the proof follows as in Lemma \ref{n=4m+3k=3}.
\end{proof}

\section{Case D: $k\equiv 1$ (mod $4$) and $n\equiv 1$ (mod $4$)}

In this section we construct a Heffter array $H(n;k)$ where $n=4m+1$ and $k=4p+1$, with $k<n$ and hence $m\geq 2$. The case $H(9;5)$ is given in the Appendix and so henceforth we assume $m\geq 3$.

We begin with $k=5$ construct an $n\times n$ array for which  the sum of each  row and column is $2nk+1=40m+11$.
First we give an example $H(17;5)$ of our construction. (Note a Hamilton cycle $H$ has been included for the case $k>5$).

\begin{center}
{\small $
\renewcommand{\arraycolsep}{2pt}
\begin{array}{|c|c|c|c|c|c|c|c|c|c|c|c|c|c|c|c|c|}
\multicolumn{15}{c}{H(17;5)}\\
\hline
85 & $-$27 &   & 11 &   &   &   & 50 &   & H  & H  &   &   &   &   &   & 52 \\ \hline
68 & 84 & $-$28 &   & 12 &   &   &   & 35 &   & H  & H  &   &   &   &   &   \\ \hline
 & 67 & 83 & $-$29 &   & 13 &   &   &   &   &   & H  & H  &   &   & 37 &   \\ \hline
 &   & 66 & 82 & $-$30 &   & 14 &   &   &   & 39 &   & H  & H  &   &   &   \\ \hline
 &   &   & 65 & 81 & $-$31 &   & 15 &   &   &   & 41 &  & H & H  &   &   \\ \hline
 &   &   &   & 64 & 80 & $-$32 &   &   &   &   &   & 43 &   & H  & 16 & H \\ \hline
 &   &   &   &   & 63 & 79 & $-$33 &   & 17 &   &   &   & 45 &   & H  &   H \\ \hline
H & H  &   &   &   &   & 62 & 78 &   & $-$34 & 18 &   &   &   & 47 &   &   \\ \hline
$-$20 & H  & H  &   &   &   &   & 61 & 60 &   &   & 21 &   &   &   & 49 &   \\ \hline
 &   &  H & H  &   &   &   &   & 77 & 59 & $-$19 &   & 3 &   &   &   & 51 \\ \hline
36 &   &   & H  & H  &   &   &   &   & 76 & 58 & $-$22 &   & 23 &   &   &   \\ \hline
 & 38 &   &   & H  & H  &   &   &   &   & 75 & 57 & $-$4 &   & 5 &   &   \\ \hline
 &   & 40 &   &   & H  &H   &   & 25 &   &   & 74 & 56 & $-$24 &   &   &   \\ \hline
 &   &   & 42 &   &   & H  & H  &   &   &   &   & 73 & 55 & $-$6 &   & 7 \\ \hline
2 &   &   &   & 44 &   &   & H  & H  &   &   &   &   & 72 & 54 & $-$1 &   \\ \hline
 & 9 &   &   &   & 46 &   &   & H  & 53 &   &   &   &   & 71 &  H & $-$8 \\ \hline
H &   & 10 &   &   &   & 48 &   & $-$26 & H  &   &   &   &   &   & 70 & 69 \\ \hline
\end{array}
$}
\end{center}

\begin{lemma}\label{n=4m+1k=5} For $n\equiv 1(\mbox{mod }4)$, $n\geq 13$ and $k=5$ there exists a Heffter array $H(n;5)$.
\end{lemma}

\begin{proof}
Let $n=4m+1$ for $m\geq 3$.  
In every case here row and column sums will equal $40m+11$.

We give the general construction  of $A$ below. 
\begin{eqnarray*}
\begin{array}{l}
A[4m-1,4m]=-1, \quad   A[4m-1,1]=2,    \quad    A[4m,2]=2m+1, \quad 
A[4m+1,3]=2m+2,  \\  A[2m-2,4m]=4m,   \quad         A[2m-1,2m+2]=4m+1,
\quad A[2m,2m+3]=4m+2,  \\   A[2m+2,2m+3]=-(4m+3),   \quad  A[2m+1,1]=-(4m+4),
\quad A[4m+1,2m+1]=-(6m+2), \\  A[4m-3,2m+1]=6m+1,       \quad A[2m,2m+2]=-(8m+2), 
\quad A[2,2m+1]=8m+3,    \\   A[3,4m]=8m+5,            \quad A[1,2m]=12m+2,
\quad A[1,4m+1]=12m+4,     \\  A[4m,2m+2]=12m+5,        \quad A[4m+1,4m+1]=16m+5,  \\
\end{array}\end{eqnarray*}

\vspace{-.2in}\begin{center}
\begin{eqnarray}
A[i,i+3]=            2m+i+2,          &&  1\leq i\leq 2m-3, \label{k=5eq1}\\
A[i+1,i]=            16m+5-i,         &&  1\leq i\leq 2m, \label{k=5eq2}\\
A[i,i]=              20m+6-i,         &&  1\leq i\leq 2m,  \label{k=5eq3}\\
A[2m-i,2m+1-i]=     -(8m+2-i),        &&  1\leq i\leq 2m-1, \label{k=5eq4}\\
A[2m+3-i,4m+2-i]=    12m+5-2i,        &&   1\leq i\leq 2m-1,\label{k=5eq5} \\
A[2m+i,2m+i]=        14m+5-i,         &&   1\leq i\leq 2m-1,\label{k=5eq6} \\
A[2m+1+i,2m+i]=      18m+6-i,         &&  1\leq i\leq 2m, \label{k=5eq7}\\
A[4m+2-i,2m-i]=      12m+2-2i,        &&  1\leq i\leq 2m-1, \label{k=5eq8}\\
A[2m+2i,2m+2i+3]=    2i +1,           &&   1\leq i\leq m-1, \label{k=5eq9}\\
A[2m+2i+2,2m+2i+3]=   -(2i+2),        &&  1\leq i\leq m-1, \label{k=5eq10} \\
A[2m+2i-1,2m+2i+2]=   4m+2i+3,        &&  1\leq i\leq m-2, \label{k=5eq11}\\
A[2m+2i+1,2m+2i+2]=  -(4m+2i+4),      &&   1\leq i\leq m-2.\label{k=5eq12}
\end{eqnarray} 
\end{center}
We note that  the support of $A$ contains 
$3,\dots, 2m$ by \eqref{k=5eq9} and \eqref{k=5eq10}, 
$2m+3,\dots, 4m-1$ by  \eqref{k=5eq1},  
$4m+5,\dots, 6m$ by \eqref{k=5eq11} and \eqref{k=5eq12}, 
$6m+3,\dots, 8m+1$ by \eqref{k=5eq4}, 
$8m+4,8m+6,8m+7,\dots,12m,12m+1,12m+3$ by \eqref{k=5eq5} and \eqref{k=5eq8}, 
$12m+6,\dots, 16m+4$ by \eqref{k=5eq6} and \eqref{k=5eq2}, 
$16m+6,\dots,20m+5$ by \eqref{k=5eq7} and \eqref{k=5eq3}. It follows that the support of $A$ is $\{1,\dots, 20m+5\}$ as required.

To verify the sum of each row and column is $40m+11$ we begin by noting that,  respectively, \eqref{k=5eq1}, \eqref{k=5eq2}, \eqref{k=5eq3}, \eqref{k=5eq4} and \eqref{k=5eq5} give the  sum for row $r$ (where $4\leq r\leq 2m-3$) and \eqref{k=5eq1}, \eqref{k=5eq2}, \eqref{k=5eq3}, \eqref{k=5eq4} and \eqref{k=5eq8}  give the sum for column $c$ (where $4\leq c\leq 2m-1$) as:
\begin{center}$
(2m+r+2)+  (16m+6-r)+      (20m+6-r)      -(6m+r+2)    + (8m-1+2r)       =40m+11$
\end{center} and
\begin{center}$
(2m+c-1)+    (16m+5-c)+             (20m+6-c)      -(6m+1+c)  + (8m+2+2c)    =40m+11.
$\end{center}

For row $2m+r$, where $3\leq r\leq 2m-4$ (or $r=2m-2$) and column $2m+c$, where $4\leq c\leq 2m-1$, we argue as follows.

Respectively  \eqref{k=5eq6}, \eqref{k=5eq7} and \eqref{k=5eq8} give a partial sum of $40m+10$  for  row $r$ and \eqref{k=5eq5}, \eqref{k=5eq6} and \eqref{k=5eq7} give a partial sum of $40m+12$ for column $c$:
\begin{eqnarray}
(14m+5-r)  +      (18m+7-r)+         (8m-2+2r)           &=&40m+10,\label{row}\\
(8m+1+2c)+ (14m+5-c)+           (18m+6-c)                 &=&40m+12.\label{col}
\end{eqnarray}

Next \eqref{k=5eq9} and  \eqref{k=5eq10} imply that if $r$ is even,  
row $2m+r$ contain the entries $r+1$ and $-r$ giving a partial sum of $1$. If $r$ is odd, \eqref{k=5eq11} and  \eqref{k=5eq12} imply that row $2m+r$ contains the entries $4m+r+4$ and $-(4m+r+3)$ also giving a partial sum of $1$. Adding this to the partial sum in \eqref{row} we get an overall row sum of $40m+11$.
Then \eqref{k=5eq9} and  \eqref{k=5eq10} imply that if $c$ is odd, column $2m+c$ contain the entries $c-2$ and $-(c-1)$ giving a partial sum of $-1$. If $c$ is even, \eqref{k=5eq11} and  \eqref{k=5eq12} imply that column $2m+c$ contain the entries $4m+c+1$ and $-(4m+c+2)$ also giving a partial sum of $-1$. Adding this to the partial sum in \eqref{col} we get a column sum of $40m+11$.

The remaining rows
and columns 
can be checked individually to complete the proof that all rows and columns sum to $40m+11$.
Thus we have the required $H(4m+1;5)$.
\end{proof}

Next, with care, we add up to $2(m-2)$ Hamilton cycles to obtain an $H(4m+1;4p+5)$ where $p\leq m-2$.

\begin{theorem}\label{n=4m+1k=4p+5} For $n\equiv 1(\mbox{mod }4)$ and $k\equiv 1(\mbox{mod }4)$ there exists a Heffter array $H(n;k)$, where $k<n$.
\end{theorem}

\begin{proof} Let $n=4m+1$. The Heffter array $H(9;5)$ is given in the Appendix. Otherwise, an $H(n;5)$ exists by Lemma \ref{n=4m+1k=5}. 
This was labeled $A$ in the proof of that lemma.
Observe that the occupied cells of $A$ are a subset of the union of diagonals $${\mathcal D}:=D_{n-3}\cup D_{n-2}\cup D_{n-1}\cup D_{0}\cup D_{1}\cup D_4\cup D_{2m-4}\cup D_{2m-2}\cup D_{2m-1}\cup D_{2m}\cup D_{2m+2}.$$

For each $m\geq 3$, the following set of cells is a subset of  ${\mathcal D}$, that does {\em not} intersect  $A$ and forms a Hamilton cycle $H$:
$$\begin{array}{l}
\{(i,2m+1+i),(i,2m+2+i)\mid 1\leq i\leq 2m-3\} 
\cup \{(2m-2,4m-1),(2m-2,4m+1)\} \\
\cup\{(2m-1,4m+1),(2m-1,4m),(4m,4m),(4m,2m+1)\} \\
\cup \{(i,i-2m+1),(i,i-2m+2)\mid 2m\leq i\leq 4m-1\} \cup \{(4m+1,1),(4m+1,2m+2)\}. \\
\end{array}$$
(See the array $H(17;5)$ above for an example.)  Thus there exists $4m+1-11=4(m-2)-2$ diagonals that do not intersect $A\cup H$ and so it is possible to construct  $2m-5$ disjoint Hamilton cycles by pairing empty diagonals that are either distance 1 or 2 apart.

For $m\geq 6$ a possible  pairing of diagonals is:
\begin{eqnarray*}
\{D_2,D_3\}; \quad  \{D_{2m-5},D_{2m-3}\}; \quad \{D_{2m+1},D_{2m+3}\};\\
\{D_{5+2i},D_{6+2i}\},\quad  0\leq i\leq m-6; \quad  \{D_{2m+4+2i},D_{2m+5+2i}\}, \quad  0\leq i\leq m-4.
\end{eqnarray*}

When $m=4$ we  pair the diagonals as $\{D_2,D_3\};\{D_9,D_{11}\};\{D_{12},D_{13}\}$ and when $m=5$ we  pair the diagonals as
 $\{D_2,D_3\};\{D_5,D_7\};\{D_{11},D_{13}\};\{D_{14},D_{15}\}; \{D_{16},D_{17}\}.$

Together with $H$ this gives a total of $2m-4$ Hamilton cycles. Thus applying Theorem \ref{crucial} recursively, we can form a Heffter array for each $k$ such that $k\equiv 1$ (mod $4$) and $k\leq n-4$.
\end{proof}

\section{Case E: $k\equiv 1$ (mod $4$) and $n\equiv 2$ (mod $4$)}

In this section we construct a Heffter array $H(n;k)$ where $n=4m+2$ and $k=4p+1$, where $k<n$.

\renewcommand{\tabcolsep}{2pt}
We demonstrate the following construction in the case $m=4$ with an example of an $H(18;5)$; the cycles $H$ and $K$ will be needed later for the case $k>5$. 

\begin{center}
{\small $
\renewcommand{\arraycolsep}{2pt}
\begin{array}{|c|c|c|c|c|c|c|c|c|c|c|c|c|c|c|c|c|c|}
\multicolumn{18}{c}{H(18;5)}\\
\hline
2 & & & & & 51 & 20 & $-38$ & $-35$ & & & & & & & & &  \\
\hline
$-40$ & 4 & & & & & 53 & 19 & $-36$ & & & & & & & & & \\
\hline
$-42$ & -28 & 6 & & & & & 37 & 27 & & & & & & & & & \\
\hline
39 & $-44$ & $-$29 & 8 & & & & & 26 & & & & & & & & &\\
\hline
41 & 25 & $-$46 & $-$30 & 10 & & & & & & & & & & & & & \\
\hline
& 43 & 24 & $-$48 & $-$31 & 12 & & & & & & & & & & & & \\
\hline
& & 45 & 23 & $-$50 & $-$32 & 14 & & & & & & & & & & & \\
\hline
& & & 47 & 22 & $-$52 & $-$33 & 16 & & & & & & & & & & \\
\hline
& & & & 49 & 21 & $-$54 & $-$34 & 18 & & & & & & & & & \\
\hline
& & & & & & & & & 1 & 81 & K & H & 69 & $-$60 & K & 90 & H \\
\hline
& & & & & & & & & H & 3 & 80 & K & H & 70 & $-$61 & K & 89 \\
\hline
& & & & & & & & & 88 & H & 5 & 79 & K & H & 71 & $-$62 & K \\
\hline
& & & & & & & & & 72 & 87 & H & 7 & 78 & K & H & K & $-$63 \\
\hline
& & & & & & & & & $-$55 & K & 86 & H & 9 & 77 & K & H & 64 \\
\hline
& & & & & & & & & K & $-$56 & K & 85 & H & 11 & 76 & 65 & H \\
\hline
& & & & & & & & & 75 & 66 & $-$57 & K & 84 & H & 13 & H & K \\
\hline
& & & & & & & & & K & H & 67 & $-$58 & K & 83 & H & 15 & 74 \\
\hline
& & & & & & & & & H & K & H & 68 & $-$59 & K & 82 & 73 & 17 \\
\hline
\end{array}$}
\end{center}

\begin{lemma}\label{n=4m+2k=5} For $n\equiv 2\ (\mbox{mod }4)$, $n\geq 6$ there exists a Heffter array $H(n;5)$.
\end{lemma}

\begin{proof} Let $n=4m+2$. 
 The case $H(n;k)=H(6;5)$ is given in the Appendix and so henceforth we assume $m\geq 2$.
Our Heffter array will be the concatenation of an array $A_0$ in the first $2m+1$ rows and an array $A_1$ in the last $2m+1$ rows and columns. The row and column sums of $A_0$ will be $0$ and the row and column sums of $A_1$ 
 will be $2nk+1$. 

In our definition of $A_0$, rows and columns are calculated modulo $2m+1$ rather than $n$. 
We begin by defining a $(2m+1)\times (2m+1)$ array $A_0^\prime$ which has  support
$$\{2,4,\dots,4m+2\}\cup \{4m+3,4m+4,\dots ,12m+6\}$$ 
and for which each row sums to $0$ and each column sums to $0$ except for columns $1$ and $2m+1$, which sum
to $-(2m+1)$ and $2m+1$, respectively. 
Then we swap  the entry $-(8m+4)$ in cell $(2,1)$ with the entry $-(8m+8)$ in cell $(2,2m+1)$ and
swap the entry $6m+2$ in cell $(4,1)$ with the entry $8m+7$ in cell $(4,2m+1)$. The result will be an array $A_0$ defined on row and column set $\{1,2,\dots ,2m+1\}$ with row and column sums equal to $0$.

To this end, for $1\leq i\leq 2m+1$ let
$$\begin{array}{c}
A_0^\prime[i,i]=2i; \quad A_0^\prime[3-i,2m+1-i]=4m+2+i; \quad A_0^\prime[2+i,1+i] =-(6m+3+i); \\
A_0^\prime[2+i,i-2]=8m+3+2i; \quad  A_0^\prime[i,i-2]=-(8m+4+2i).
\end{array}$$

Now let
$$
A_0[r,c]=\left\{
\begin{array}{ll}
-(8m+8),&(r,c)=(2,1),\\
-(8m+4),&(r,c)=(2,2m+1),\\
8m+7,&(r,c)=(4,1),\\
6m+2,&(r,c)=(4,2m+1),\\
A_0^\prime[r,c],&\mbox{otherwise.}
\end{array}.\right.$$

The case $m=4$ is illustrated in the example $H(18;5)$ given above. 
It will be useful to note that the non-empty cells of this $(2m+1)\times (2m+1)$ array  $A_0$ are a subset of the union of diagonals
$$
{\cal D}_0:=D_0\cup D_1\cup D_2\cup D_3\cup D_4.
$$

Next, we define a $(2m+1)\times (2m+1)$ array $A_1$, on row and column set $\{2m+2,2m+3,\dots ,4m+2\}$ with support
$\{1,3,\dots,4m+1\}\cup \{12m+7,12m+8,\dots ,16m+8\}\cup \{kn-4m-1,kn-4m,\dots ,kn\}$.
For the case $k=5$, observe that $\{kn-4m-1,\dots ,kn\}=\{16m+9,\dots ,20m+10\}$. Thus when $k=5$ the support of $A_0\cup A_1$ is $\{1,\dots, 20m+10\}$ as required.  
Each row and column of $A_1$ will sum to $2nk+1$.
We first give $A_1$ for the cases $m=2$ and $m=3$ separately; then we present the general formula.

\begin{center}
{\small $
\renewcommand{\arraycolsep}{1pt}
\begin{array}{|c|c|c|c|c|}
\multicolumn{5}{c}{A_1\ (m=2,n=10)}\\
\hline
1 & 10k-4 & -31 & 10k-2 & 37 \\
\hline
39 & 3 & 10k & -34 & 10k-7 \\
\hline
-33 & 40 & 5 & 10k-8 & 10k-3 \\
\hline
10k-1 & 10k-6 & 36 & 7 & -35 \\
\hline
10k-5 & -32 & 10k-9 & 38 & 9 \\
\hline
\end{array}$
}
\end{center}

\begin{center}
{\small $
\renewcommand{\arraycolsep}{1pt}
\begin{array}{|c|c|c|c|c|c|c|}
\multicolumn{7}{c}{A_1\ (m=3,n=14)}\\
\hline
1 & 14k & & & -43 & 14k-7 & 50 \\
\hline
14k-1 & 3 & 51 & 14k-8 &  & -44 &  \\
\hline
& -45 & 5 & 14k-2 & 55 & & 14k-12 \\
\hline
& & 14k-3 & 7 & 14k-11 & 54 & -46 \\
\hline
56 & & -47 & & 9 & 14k-13 & 14k-4 \\
\hline
14k-6 & 53 &  & -48 & 14k-9 & 11 & \\
\hline
-49 & 14k-10 & 14k-5 & 52 &  & & 13 \\
\hline
\end{array}$
}
\end{center}

Otherwise $m\geq 4$ and we define $A_1$ as follows.
The rows and columns are defined modulo $2m+1$ rather than modulo $n$.
To construct the overall Heffter array, the array $A_1$ is then shifted by adding $2m+1$ (as an integer) to each row and column.

\begin{eqnarray}
A_1[4,1]= 14m+8; & A_1[5,2m+1]=14m+9;\label{k=5eq18}\\
A_1[6,2m]=16m+8; & A_1[2m-1,1]=kn-4m+1;\label{k=5eq19}\\
A_1[2m,2m+1]=kn-4m; & A_1[2m+1,2m]=kn-4m-1;\label{k=5eq20}\\ 
A_1[i,i]=2i-1, & 1\leq i\leq 2m+1;\label{k=5eq13} \\
A_1[i,2m-1+i]=kn+1-i, &1\leq i\leq 2m+1;\label{k=5eq14} \\
A_1[i,i+1]=kn-2m-i, & 1\leq i\leq 2m-2;\label{k=5eq15} \\
A_1[4+i,i]=-(12m+6+i), &1\leq i\leq 2m+1;\label{k=5eq16} \\
 A_1[6+i,1+i]=14m+9+i,&1\leq i\leq 2m-2.\label{k=5eq17}
\end{eqnarray}

We note that the support for $A_1$ is the union of the sets:

\begin{center}
\begin{tabular}{l}
 $\{1,3,5,\dots, 4m+1\}$ (by \eqref{k=5eq13}), $\{12m+7,\dots,14m+7\}$ (by \eqref{k=5eq16}), \\
$\{14m+8,14m+9\}$ (by \eqref{k=5eq18}),  
$\{14m+10,\dots, 16m+7\}$ (by \eqref{k=5eq17}), $\{16m+8\}$ (by \eqref{k=5eq19}), \\ $\{kn-4m-1,kn-4m,kn-4m+1\}$ (by \eqref{k=5eq19},\ \eqref{k=5eq20}), \\
$\{kn-4m+2,\dots, kn-2m-1\}$ (by \eqref{k=5eq15}) and $\{kn-2m,\dots,kn\}$ (by \eqref{k=5eq14}).
\end{tabular}
\end{center}

We next check the row and column sums. 
For row $r$ in the range $7$ to $2m+1$, \eqref{k=5eq14} and \eqref{k=5eq15} give a partial sum of 
$$(kn+1-r)+(kn-2m-r)=2kn-2m+1-2r,$$  
while \eqref{k=5eq16} and \eqref{k=5eq17} give a partial sum of $(-12m-6-(r-4))+(14m+9+(r-6))=2m+1.$ 
Now combined with \eqref{k=5eq13} the sum of these rows is
\begin{eqnarray*}
(2nk-2m+1-2r)+(2m+1)+(2r-1)=2nk+1,
\end{eqnarray*}
as required.

For column $c$ in the range $2$ to $2m-1$, \eqref{k=5eq14} and \eqref{k=5eq15} give a partial sum of $$(kn+1-(c+2))+(kn-2m-(c-1))=2kn-2m-2c,$$  while \eqref{k=5eq16} and \eqref{k=5eq17} give a partial sum of 
$(-12m-6-c)+(14m+9+(c-1))=2m+2.$ Now combined with \eqref{k=5eq13} the sum of these columns is
\begin{eqnarray*}
(2nk-2m-2c)+(2m+2)+(2c-1)=2nk+1.
\end{eqnarray*}
The sum of the remaining rows and columns can be calculated individually and overall the rows and columns of $A_1$ sum to $2nk+1$ as required.
Thus the concatenation of $A_0$ with $A_1$ gives an $H(4m+2;5)$.
\end{proof}

\begin{theorem}\label{n=4m+2k=4p+5} For $n\equiv 2\ (\mbox{mod }4)$, $n\geq 6$ and $k\equiv 1\ (\mbox{mod }4)$ there exists a Heffter array $H(n;k)$, where $k<n$.
\end{theorem}

\begin{proof} 
Let $n=4m+2$ and $k=4p+1$. 
A $H(n;5)$ exists by Lemma \ref{n=4m+2k=5}.  
Otherwise, $k\geq 9$ and $m\geq 2$.  
We take the array $A=A_0\cup A_1$ from the previous lemma.

We will construct $m-2$ cycles of length $n$ (that is, on $2(2m+1)=n$ cells) in the upper left-hand ($A_0$) and lower right-hand ($A_1$) quadrants,
and $m$ further cycles of length $n$ in each of the remaining quadrants. Together these
form $2m-2$ disjoint $2$-factors.

From the proof of Lemma \ref{n=4m+2k=5}, within $A_0$ there are $2m+1-5=2(m-2)$ empty diagonals, which we take in pairs to obtain $m-2$ cycles of length $n$.
Next take the intersection of the last $2m+1$ rows and columns, this is the $(2m+1)\times (2m+1)$ subarray that contains $A_1$. We will refer to diagonals within that subarray only, recalling that the rows and columns are calculated  modulo $2m+1$.
We aim to find $m-2$ cycles of length $n$ from this subarray. The case $m=2$ is trivial and for the case $m=3$, observe that the empty cells of $A_1$ in the previous lemma form a cycle of length $14$. 
Otherwise $m\geq 4$ and the array $A_1$ occupies diagonals $D_0,D_1,D_2,D_3,D_4,D_5,D_7,D_{2m-2}$ and $D_{2m}$.

Next take the following cells
\begin{eqnarray*}
H&=&(\{(i+1,i),(2m-2+i,i)\mid 1\leq i\leq 2m+1\}\setminus\{(2m-1,1),(2m+1,2m)\})\\
&&\cup \{(2m-1,2m),(2m+1,1)\}\\
K&=&(\{(3+i,i),(7+i,i)\mid 1\leq i\leq 2m+1\}\setminus \{(4,1),(6,2m)\})\cup\{(4,2m),(6,1)\}.
\end{eqnarray*}
In the example $H(18;5)$ above these cycles are shown in cells marked by $H$ and $K$, respectively.
Observe that $H$ and $K$ form two  cycles  of length $4m+2$ disjoint from $A_1$ but  are a subset of
\begin{eqnarray*}
D_1\cup D_{2m-2}\cup D_{2m}\cup D_3\cup D_7\cup D_5.
\end{eqnarray*}

Thus there exists $2m+1-9=2(m-4)$ diagonals that do not intersect $A_1\cup H\cup K$.
For $m\geq 6$, we can thus form $m-4$ cycles of length $n$ by taking pairs of diagonals: 
$$\{D_6,D_8\}; \{D_{2m-3},D_{2m-1}\}; \{D_{9+2i},D_{10+2i}\},\quad  0\leq i\leq m-7,$$
and when $m=5$, $2m+1=11$ so we get one cycle of length $n$ by taking the diagonals $D_6$ and $D_9$.
Thus with $H$ and $K$ we have $m-2$ cycles of length $n$ that are disjoint from $A_1$ and each other; together these form 
$m-2$ $2$-factors, each consisting of two cycles of length $n$.

We can create a further
 $m$  cycles of length $n$ in each of the remaining quadrants, as these cells are all empty. Altogether we have $2m-2$ disjoint $2$-factors. Thus by Lemma \ref{twofactor},
we can fill $4(p-1)$ cells in each row and column with support
$\{16m+9,16m+10,\dots , kn-4m-2\}$ without changing the row and column sums, where $k=4p+1$.
Thus there exists an $H(4m+2;4p+1)$ Heffter array for each $m\geq 2$ and $p\leq m$. 
\end{proof}

\section{Appendix}

{\bf Case B:}
{\small
$$\begin{array}{ccc}
H(7;3)&&B(11)\\
$$\begin{array}{|c|c|c|c|c|c|c|}
\hline
15&	-13&	-2&&&&	\\
\hline							
-11&	14&		&-3&&&\\
\hline							
-4&		&-8&	12&&&\\
\hline							
  &-1&	10&	-9&&&	\\
\hline		
&&&&								5&	21&	17\\
\hline
&&&&								18&	6&	19\\
\hline
&&&&								20&	16&	7\\
\hline

\end{array}$$
&&
$$\begin{array}{|c|c|c|c|c|c|c|c|}
\hline
-1&	18&	-17&&&&&\\
\hline							
24&	-2&		&-22&&&&\\
\hline							
-23&		&-3&	26&&&&\\
\hline							
  &-16&	20&	-4&&&&\\
\hline						
&&&&				19&	-8&	-11& \\
\hline			
&&&&				-9&	21&		&-12\\
\hline		
&&&&				-10&		&25	&-15\\
\hline		
&&&&				&	-13&	-14&	27\\
\hline		

\end{array}$$
\end{array}$$
}
{\footnotesize
$$\begin{array}{|c|c|c|c|c|c|c|c|c|c|c|c|}
\multicolumn{12}{c}{B(15)} \\ 
\hline
1&	-36&	35&&&&&&&&&\\
\hline												
-34&	-3&		&37&&&&&&&&\\
\hline											
33&		&	&-22&	-11&&&&&&&\\
\hline										
&	39&	-21&	&		&-18&&&&&&\\
\hline									
&		&-14&	&	&	-12&	26&&&&&\\
\hline								
&		&	&-15&	-17&	&		&32&&&&\\
\hline							
&		&	&	&28&	&	&	-19&	-9&&&\\
\hline						
&		&	&	&	&30&	-10&	&		&-20&&\\
\hline					
&		&	&	&	&	&-16&		&	&24&	-8&\\
\hline				
&		&	&	&	&	&	&-13&	38&		&	&-25\\
\hline			
&		&	&	&	&	&	&	&-29&		&31&	-2\\
\hline
&		&	&	&	&	&	&	&	&-4&	-23&	27\\
\hline

\end{array}$$
}

\renewcommand{\arraycolsep}{1pt}

\noindent{\bf Case C:}


\renewcommand{\arraycolsep}{3pt}
\bigskip

{\small        
$$\begin{array}{|c|c|c|c|c|c|}
\multicolumn{6}{c}{H(6;3)}\\
\hline
-1&	-16&	&	&	&	17\\
\hline
-11&   &-4	&	&	&15\\
\hline
12&		&-9	&-3	&	&\\
\hline
	&-2	&	&10	&-8	&\\
\hline
	&	&13	&-7	&-6	&\\
\hline
	&18	&	&	&14	&5\\
\hline
\end{array} \hspace{.5in}
\begin{array}{|c|c|c|c|c|c|c|}
\multicolumn{7}{c}{B(10)}\\
\hline
1&22&	-23&&&&\\
\hline							
17&	2&&-19&&&\\
\hline						
-18&&&15&3&&\\
\hline					
&-24&14&	&&10&\\
\hline				
&&9	&&&11&-20\\
\hline		
&&&	4&-16&	&12\\
\hline
&&&&13&	-21&	8\\
\hline
\end{array}$$

{\scriptsize
$$\begin{array}{|c|c|c|c|c|c|c|c|c|c|c|}
\multicolumn{11}{c}{B(14)}\\
\hline
-34&	-1&	35&&&&&&&&\\
\hline								
-2&	24&		&-22&&&&&&&\\
\hline							
36&	&	-32&	-4&&&&&&&\\
\hline							
&	-23	&-3&	26&&&&&&&\\
\hline							
& &	&			&-20&	28&	-8&&&&\\
\hline				
&	&	&		&30&	-9&	-21&&&&\\
\hline				
&   &   &		&-10&	-19&	29&&&&\\
\hline				
&	&	&		&	&	&	&-11&	27&	-16&\\
\hline	
&	&	&		&	&	&	&25&	-12&		&-13\\
\hline
&	&	&		&	&	&	&-14&	&	-17&	31\\
\hline
&	&	&		&	&	&	&	&-15	&33&	-18\\
\hline
\end{array}$$
}

{\tiny
 $$\begin{array}{|c|c|c|c|c|c|c|c|c|c|c|c|c|c|c|c|}
\multicolumn{14}{c}{ B(18)}\\
\hline
1	&21	&-22&&&&&&&&&&&&\\
\hline										
-36	&-4	&	&40&&&&&&&&&&&\\
\hline									
35	&	&	&-12&	-23	&	&	&	&	&	&	&	&		&   &   \\
\hline
	&-17	&33	&	&	&-16&	&	&	&	&	&	&		&	&   \\
\hline
	&	&-11	&	&	&42	&-31&	&	&	&	&	&		&	&   \\
\hline
	&	&	&-28	&41	&	&	&-13&	&	&	&	&		&	&   \\
\hline
	&	&	&	&-18	&	&	&43	&-25&	&	&	&		&	&   \\
\hline
	&	&	&	&	&-26	&45	&	&	&-19&	&	&		&	&   \\
\hline
	&	&	&	&	&	&-14	&	&	&34	&-20&	&		&	&   \\
\hline
	&	&	&	&	&	&	&-30	&-2	&	&	&32&		&	&    \\
\hline
	&	&	&	&	&	&	&	&27	&	&-24	&-3&		&	&    \\
\hline
	&	&	&	&	&	&	&	&	&-15	&44	&-29&	&	&	\\
\hline
	&	&	&	&	&	&	&	&	&	&	&	&-8	&46	&-38\\
\hline
	&	&	&	&	&	&	&	&	&	&	&	&-39	&-9	&48\\
\hline
	&	&	&	&	&	&	&	&	&	&	&	&47	&-37	&-10\\
\hline
\end{array}$$
}
{\tiny 
 $$\begin{array}{|c|c|c|c|c|c|c|c|c|c|c|c|c|c|c|c|c|c|c|}
\multicolumn{18}{c}{ B(22)}\\
\hline
    1&-36&35&  &    &   &   &   & & & & & & & & & & & \\
\hline																
  -30&-4&	&34&    &   &   &   & & & & & & & & & & & \\
\hline															
   29&	&	&-3&-26 &   &   &   & & & & & & & & & & & \\
\hline														
	&40 & -2&	&	&-38&   &   & & & & & & & & & & & \\
\hline													
	&	&-33&	&	&54& -21&   & & & & & & & & & & & \\
\hline												
	&	&	&-31& 48&	&	&-17& & & & & & & & & & & \\
\hline											
 	&	&	&	&-22&	&	&42 &-20&	&	&	&	&	&	&	&	&	&	\\
\hline
	&	&	&	&	&-16&53 &	&	&-37&	&	&	&	&	&	&	&	&	\\
\hline
	&	&	&	&	&	&-32&	&	&-15&47 &	&	&	&	&	&	&	&	\\
\hline
	&	&	&	&	&	&	&-25&39 &	&	&-14&	&	&	&	&	&	&   \\
\hline
	&	&	&	&	&	&	&	&-19&	&	&-27&46 &	&	&	&	&	&   \\
\hline
	&	&	&	&	&	&	&	&	&52	&-24&	&-28&	&	&	&	&	&   \\
\hline
	&	&	&	&	&	&	&	&	&	&-23&41	&-18&	&	&	&	&	&    \\
\hline
	&	&	&	&	&	&	&	&	&	&	&	&   &-11& 55& -44&&&\\
\hline
	&	&	&	&	&	&	&	&	&	&	&	&    &-45& -12&57&&&\\
\hline
	&	&	&	&	&	&	&	&	&	&	&	&    & 56& -43&-13&&&\\
\hline
	&	&	&	&	&	&	&	&	&	&	&	 &   &   &   &   &-8	&58	&-50\\
\hline
	&	&	&	&	&	&	&	&	&	&	&	&    &   &   &   &-51	&-9	&60\\
\hline
	&	&	&	&	&	&	&	&	&	&	&	&    &   &   &   &59	&-49	&-10\\
\hline
\end{array}$$
}

{\tiny
 $$\begin{array}{|c|c|c|c|c|c|c|c|c|c|c|c|c|c|c|c|c|c|c|c|c|c|c|c|}
\multicolumn{23}{c}{ B(26)}\\
\hline
-64	&1&63 &	&	&	&	&	&	&	&	&	&   &   & 	&	&	&	&   &   & 	&	&	\\
\hline										
	&65	&-16&-49&	&	&	&	&	&	&	&	&   &   & 	&	&	&	&   &   & 	&	&	\\
\hline										
-3	&	&	&53	&-50	&	&	&	&	&	&	&	&   &   & 	&	&	&	&   &   & 	&	&	\\
\hline									
	&-66	&	&-4	&	&70  &	&	&	&	&	&	&   &   & 	&	&	&	&   &   & 	&	&	\\
\hline										
	&	&-47&	&&-22	&69	&	&	&	&	&	&   &   & 	&	&	&	&   &   & 	&	&	\\
\hline									
	&	&	&	&68&-48 &	&-20&	&	&	&	&	&   &	&	&	&   &   &   &	&	&   \\
\hline										
	&	&	&	&-18	& &	&59	&-41&	&	&	&	&   &	&	&	&   &   &   &	&	&	\\
\hline										
	&	&	&	&	&	&-52&	&-10&62	&	&	&	&   &	&	&	&   &   &   &	&	&   \\
\hline									
	&	&	&	&	&	&-17&	&51	&-34&	&	&	&   &	&	&	&   &   &   &	&	&   \\
\hline										
67	&	&	&	&	&	&	&-39&	&-28&	&	&	&   &	&	&	&   &   &   &	&	&	\\
\hline										
	&	&	&	&	&	&	&	&	&	&-29&61 &-32&	&	&   &	&	&	&   &   &   &   \\
\hline								
	&	&	&	&	&	&	&	&	&	&-42&-30&72 &	&	&	&	&	&	&   &   &   &   \\
\hline						
	&	&	&	&	&	&	&	&	&	&71	&-31&-40&	&	&	&	&	&	&	&   &   & 	\\
\hline				
	&	&	&	&	&	&	&	&	&	&	&	&	&-2&60	&-58&	&	&	&	&   &   &   \\
\hline			
	&	&	&	&	&	&	&	&	&	&	&	&	&23&    &15	&-38&	&	&	&	&	&   \\
\hline	
	&	&	&	&	&	&	&	&	&	&	&	&	&-21&-33&	&	&54	&	&	&	&	&   \\
\hline
	&	&	&	&	&	&	&	&	&	&	&	&	&   &-27&	&	&-8	&35	&	&	&	&	\\
\hline
	&	&	&	&	&	&	&	&	&	&	&	&	&	&	&43	&14	&	&	&-57&	&	&	\\
\hline
	&	&	&	&	&	&	&	&	&	&	&	&	&	&	&	&24	&	&	&13	&-37&	&	\\
\hline
	&	&	&	&	&	&	&	&	&	&	&	&	&	&	&	&	&-46&-9	&	&	&55	&   \\
\hline
	&	&	&	&	&	&	&	&	&	&	&	&	&	&	&	&	&	&-26&	&	&-19& 45\\
\hline
	&	&	&	&	&	&	&	&	&	&	&	&	&	&	&	&	&	&	&44	&12	&	&-56\\
\hline
	&	&	&	&	&	&	&	&	&	&	&	&	&	&	&	&	&	&	&	&25 &-36& 11\\
\hline
\end{array}$$
}

\renewcommand{\arraycolsep}{1pt}
{\tiny
$$\begin{array}{|c|c|c|c|c|c|c|c|c|c|c|c|c|c|c|c|c|c|c|c|c|c|c|c|c|c|c|c|c|c|c|}
\multicolumn{27}{c}{ H(30,3)}\\
\hline
1	&-62&61 &K	&	&	&	&	&	&	&	&	&	&	&	&	&   &	&	&	&	&	&	&	&	&	&	&K	& H & H\\\hline
-77	&K	&H	&-3 &80	&	&	&	&	&	&	&	&	&	&	&	&	&	&	&	&	&	&	&	&	&	&	&	&K	& H     \\\hline
	&83	&-12&H	&H	&-71&	&	&	&	&	&	&	&	&K	&	&	&	&	&	&	&	&	&	&	&	&	&	&	&K		\\\hline
K	&-21&K	&-60&H	&H	&81	&	&	&	&	&	&	&	&	&	&	&	&	&	&	&	&	&	&	&	&	&	&	&		\\\hline
	&H	&-49&K	&K	&75	&H	&-26&	&	&	&	&	&	&	&	&	&	&	&	&	&	&	&	&	&	&	&	&	&		\\\hline	
	&	&H	&63	&K	&-4	&K	&H	&-59&	&	&	&	&	&	&	&	&	&	&	&	&	&	&	&	&	&	&	&	&		\\\hline	
	&	&	&H	&-22&K	&-57&K	&79	&H	&	&	&	&	&	&	&	&	&	&	&	&	&	&	&	&	&	&	&	&		\\\hline
	&	&	&	&-58&H	&K	&K  &-20&78	&H	&	&	&	&	&	&	&	&	&	&	&	&	&	&	&	&	&	&	&		\\\hline
	&	&	&	&	&K	&-24&74	&K	&H	&-50&H	&	&	&	&	&	&	&	&	&	&	&	&	&	&	&	&	&	&		\\\hline
	&	&	&	&	&	&H	&-48&H	&K	&-34&82	&K	&	&	&	&	&	&	&	&	&	&	&	&	&	&	&	&	&		\\\hline
	&	&	&	&	&	&	&H	&K	&-38&H	&-35&73	&K	&	&	&	&	&	&	&	&	&	&	&	&	&	&	&	&		\\\hline
76	&	&	&	&	&	&	&	&H	&-40&K	&K	&-36&H	&	&	&	&	&	&	&	&	&	&	&	&	&	&	&	&		\\\hline
	&	&	&	&	&	&	&	&	&K  &84	&-47&-37&H	&H	&K	&	&	&	&	&	&	&	&	&	&	&	&	&	&		\\\hline
	&	&	&	&	&	&	&	&	&	&K  & H	&H  &2  &70 &-72&K	&	&	&	&	&	&	&	&	&	&	&	&	&		\\\hline
	&	&	&	&	&	&	&	&	&	&	& K &H  &-27&-19&H	&46	&K	&	&	&	&	&	&	&	&	&	&	&	&		\\\hline
	&	&	&	&	&	&	&	&	&	&	&   &K  &25 &H  &39	&H  &-64&K	&	&	&	&	&	&	&	&	&	&	&		\\\hline
	&	&	&	&	&	&	&	&	&	&	&	&	&K	&-51&H	&-18&H  &69 &K	&	&	&	&	&	&	&	&	&	&		\\\hline		
	&	&	&	&	&	&	&	&	&	&	&	&	&	&K	&33	&H  &8	&H  &-41&K	&	&	&	&	&	&	&	&	&		\\\hline			
	&	&	&	&	&	&	&	&	&	&	&	&	&	&	&K	&-28&H	&-17&H	&45	&K	&	&	&	&	&	&	&	&		\\\hline		
	&	&	&	&	&	&	&	&	&	&	&	&	&	&	&	&K	&56	&H	&9	&H	&-65&K	&	&	&	&	&	&	&		\\\hline		
	&	&	&	&	&	&	&	&	&	&	&	&	&	&	&	&	&K	&-52&H	&-16&H	&68	&K	&	&	&	&	&	&		\\\hline	
	&	&	&	&	&	&	&	&	&	&	&	&	&	&	&	&	&	&K	&32	&H  &10 &H	&-42&K	&	&	&	&	&		\\\hline	
	&	&	&	&	&	&	&	&	&	&	&	&	&	&	&	&	&	&	&K	&-29&H	&-15&H	&44	&K	&	&	&	&		\\\hline	
	&	&	&	&	&	&	&	&	&	&	&	&	&	&	&	&	&	&	&	&K	&55	&H  &11	&H  &-66&K	&	&	&	   \\\hline
	&	&	&	&	&	&	&	&	&	&	&	&	&	&	&	&	&	&	&	&	&K	&-53&H	&-14&H	&67	&K	&   &   \\\hline	
	&	&	&	&	&	&	&	&	&	&	&	&	&	&	&	&	&	&	&	&	&	&K	&31	&H	&23 &-54&H	&K	&   \\\hline
	&	&	&	&	&	&	&	&	&	&	&	&	&	&	&	&	&	&	&	&	&	&	&K	&-30&43	&-13&H	&H	&K   \\\hline
K&&&&&&&&&&&&&&&&&&&&&&&&K&H&H&5&kn&kn-4\\
\hline
H&K&&&&&&&&&&&&&&&&&&&&&&&&K&H&kn-3&6&kn-2\\
\hline
H&H&K&&&&&&&&&&&&&&&&&&&&&&&&K&kn-1&	kn-5&7\\
\hline
\end{array}$$}

\renewcommand{\arraycolsep}{3pt}
{\bf Cases D and E:}
{\small 
$$\begin{array}{|c|c|c|c|c|c|c|c|c|}
\multicolumn{9}{c}{H(9;5)}\\
\hline
45 & 36 & 20 & & & & & -18 & 8 \\
\hline
-16 & 24 & 43 & 34 & & & & & 6 \\
\hline
& 44 & 35 & 22 & 7 & & -17 & & \\
\hline
& & 5 & 42 & -15 & 33 & & 26 & \\
\hline
9 & & & -10 & 32 & 41 & 19 & & \\
 \hline
 & 1 & & & 40 & -2 & 21 & 31 & \\
\hline
& & -12 & & & 23 & 30 & 39 & 11 \\
\hline
25 & & & 3 & & -4 & 38 & & 29 \\
\hline
28 & -14 & & & 27 & & & 13 & 37 \\
\hline
\end{array}
\qquad
\begin{array}{|c|c|c|c|c|c|}
\multicolumn{6}{c}{H(6;5)} \\
\hline
1 & 2 & 3 & & -25 & 19 \\
\hline
5 & 6 & 16 & 4 & & 30 \\
\hline
23 & 7 & 9 & 8 & 14 & \\
\hline
11 & & 15 & 12 & 10 & 13 \\
\hline
& 24 & 18 & 17 & 29 & -27 \\
\hline
21 & 22 & & 20 & -28 & 26 \\
\hline
\end{array}
$$
}

\bigskip

{\footnotesize
$$\begin{array}{|c|c|c|c|c|c|c|c|c|c|c|c|c|}
\multicolumn{13}{c}{H(13;9)} \\
\hline
65 & -21 & H & 9 & & 38 &  & K & K & & & H & 40 \\
\hline
52 & 64 & -22 & H & 10 & & 27 & & K & K & & &  H  \\
\hline
H & 51 & 63 & -23 & H & 10 & & 27 & & K & K & 29 & \\
\hline
& & 50 & 62 & -24 & H & H & & 31 & & K & 12 & K \\
\hline
& H & H & 49 & 61 & -25 & & 13 & & 33 & & K & K \\
\hline
K & K & & H & 48 & 60 & H & -26 & 14 & & 35 & & \\
\hline
-16 & K & K & & & 47 & 46 & H & H & 17 & & 37 & \\
\hline
& & K & K & & H & 59 & 45 & -15 & H & 3 & & 39 \\
\hline
28 & H & & K & K & & 19 & 58 & 44 & -18 & H & &  \\
\hline
& 30 & & & K & K & & H & 57 & 43 & -4 & H & 5 \\
\hline
2 & & 32 & & & K & K & & H & 56 & 42 & -1 & H  \\
\hline
& 7  & & 34 & H & & K & 41 & & H & 55 & K & -6 \\
\hline
H & & 8 & & 36 & & -20 & K & & & H & 54 & 53 \\
\hline
\end{array}$$
}

\end{document}